\documentclass[11pt]{amsart}

\usepackage{amscd,amsxtra,amssymb,mathrsfs, bbm}
\usepackage{stmaryrd}
\usepackage{amscd,amsxtra,amssymb, bbm, url}
\usepackage[new]{old-arrows}
\usepackage{enumitem}
\usepackage{hyperref}

\usepackage[paper=a4paper,left=25mm,right=25mm,top=30mm,bottom=30mm]{geometry}

\usepackage{tikz}
\usetikzlibrary{calc, arrows, positioning, shapes, fit, matrix, decorations}
\usetikzlibrary{decorations.shapes, decorations.pathreplacing}

\usepackage[all]{xy}

\newtheorem*{theorem*}{Theorem}
\newtheorem{theorem}{Theorem}[section]

\newtheorem{lemma}[theorem]{Lemma}
\newtheorem{proposition}[theorem]{Proposition}

\theoremstyle{definition}
\newtheorem{definition}[theorem]{Definition}
\newtheorem{remark}[theorem]{Remark}
\newtheorem{example}[theorem]{Example}

\numberwithin{equation}{section}

\DeclareMathOperator{\rk}{\mathrm{rk}}

\DeclareMathOperator{\cok}{\mathrm{cok}}
\renewcommand{\ker}{\mathsf{ker}}

\newcommand{\cone}{\mathsf{cone}}

\newcommand{\id}{\mathrm{id}}

\DeclareMathOperator{\Hom}{\mathsf{Hom}}

\DeclareMathOperator{\Ext}{\mathsf{Ext}}

\DeclareMathOperator{\Aut}{\mathsf{Aut}}

\DeclareMathOperator{\Ob}{\mathsf{Ob}}

\DeclareMathOperator{\lieg}{\mathfrak{g}}

\newcommand{\kk}{\mathbbm{k}}

\newcommand{\ZZ}{\mathbb{Z}}
\newcommand{\QQ}{\mathbb{Q}}

\newcommand{\FF}{\mathbb{F}}

\newcommand{\XX}{\mathbb{X}}

\newcommand{\kA}{\mathcal{A}}

\newcommand{\kB}{\mathcal{B}}

\newcommand{\kD}{\mathcal{D}}

\newcommand{\lar}{\longrightarrow}

\def\sD{\mathsf D}

\def\sH{\mathsf H}

\def\sK{\mathsf K}

\newcommand{\overr}{\overrightarrow}
\newcommand{\overl}{\overleftarrow}

\title[Double Hall algebras and derived equivalences revisited]{Double Hall algebras and derived equivalences revisited}

\author{Igor Burban}
\address{
Universität Paderborn, Institut für Mathematik, Warburger Str. 100, 33098 Paderborn
}
\email{burban@math.uni-paderborn.de}

\author{Daniel Perniok}
\address{
Universität Paderborn, Institut für Mathematik, Warburger Str. 100, 33098 Paderborn
}
\email{dperniok@math.uni-paderborn.de}

\begin{document}

\begin{abstract} Let $\kk$ be a finite field and $\kA, \kB$ be $\kk$-linear $\Ext$-finite hereditary abelian categories. A theorem of Cramer asserts that, under suitable assumptions, a derived equivalence 
$\kD^b(\kA) \!\longrightarrow\! \kD^b(\kB)$ between two such categories induces an algebra isomorphism of the corresponding double Hall algebras $\sD\sH_\kA \!\longrightarrow\! \sD\sH_\kB$. It turns out that a counting formula for certain distinguished triangles in $\kD^b(\kA)$, on which Cramer's proof relies, is incorrect in general. We give a corrected proof of Cramer's theorem which preserves the overall strategy of his approach. 
\end{abstract}

\maketitle

\renewcommand{\thefootnote}{}
\footnotetext{Date: August 25, 2026.} 
\renewcommand{\thefootnote}{\arabic{footnote}}

\section{Introduction}

Let $\kk=\FF_q$ be a finite field and $\widetilde\QQ = \QQ[\sqrt{q}]$. By a work of Ringel \cite{Ringel_HallAlgebrasAndQuantumGroups}, to any $\kk$-linear $\Ext$-finite hereditary abelian category $\kA$ over $\kk$ one can associate a $\widetilde\QQ$-algebra $\sH_\kA$ called the (extended, twisted) Hall algebra of $\kA$. If $\kA$ is a length category, then, by a work of Green, $\sH_\kA$ is a bialgebra \cite{Green_HallAlgebras} equipped with a bialgebra pairing $\sH_\kA \times \sH_\kA \stackrel{\varphi}\lar \widetilde\QQ$. Moreover, in this case $\sH_\kA$ is a Hopf algebra \cite{Xiao}. In general, $\sH_\kA$ is a topological bialgebra. Nevertheless, in this setting one can introduce the reduced Drinfeld double $\sD\sH_\kA$ of $\sH_\kA$ with respect to the pairing $\varphi$, called the double Hall algebra of $\kA$; see \cite[Appendix B]{BurbanSchiffmann_HallAlgebraEllipticCurve} for the basic properties of $\sD\sH_\kA$ in this generality. It turns out that $\sD\sH_\kA$ is well-behaved with respect to derived equivalences.

Let $\overr{\Delta}$ be a finite quiver without loops and oriented cycles,
$\diamond$ be a sink of $\overr{\Delta}$ and
$\overl{\Delta}$ be the quiver obtained from $\overr{\Delta}$ by inverting all the arrows
ending at $\diamond$ (thus, $\diamond$ is a source of $\overl{\Delta}$).
Then one has an adjoint pair of the so-called BGP-reflection
functors
$$
S_\diamond^+: \; \mathit{Rep}_{\kk}(\overr{\Delta}) \lar \mathit{Rep}_{\kk}(\overl{\Delta}) \quad \mbox{and} \quad S_{\diamond}^{-}: \; \mathit{Rep}_{\kk}(\overl{\Delta}) \lar \mathit{Rep}_{\kk}(\overr{\Delta})$$ which induce mutually inverse derived equivalences $$T_{\diamond}^{+} = RS_{\diamond}^+: \; \kD^b\bigl(\mathit{Rep}_{\kk}(\overr{\Delta})\bigr) \lar \kD^b(\mathit{Rep}_{\kk}\bigl(\overl{\Delta})\bigr)
$$
and $$
T_{\diamond}^{-} = LS_{\diamond}^-: \; \kD^b\bigl(\mathit{Rep}_{\kk}(\overl{\Delta})\bigr) \lar
\kD^b(\mathit{Rep}_{\kk}\bigl(\overr{\Delta})\bigr).
$$
Sevenhant and Van den Bergh proved in \cite{SevenhBergh} that $T_\diamond^\pm$ induce mutually inverse isomorphisms of the corresponding double Hall algebras $\sD\sH_{\overr{\Delta}} \longrightarrow \sD\sH_{\overl{\Delta}}$. Moreover, they categorify the so-called Lusztig symmetries \cite{Lusztig} of the quantized enveloping algebra of the associated Kac--Moody algebra $\lieg_{\Delta}$. This result was extended by Xiao and Yang \cite{XiaoYang} to the setting of arbitrary hereditary algebras.

Next, let $E$ be an elliptic curve over $\kk$ and $\mathit{Coh}(E)$ be the corresponding category of coherent sheaves. Then the derived category $\kD^b\bigl(\mathit{Coh}(E)\bigr)$ admits a rich group of auto-equivalences. Schiffmann and the first-named author proved in \cite{BurbanSchiffmann_HallAlgebraEllipticCurve} that any auto-equivalence of
$\kD^b\bigl(\mathit{Coh}(E)\bigr)$
induces an algebra automorphism of the corresponding double Hall algebra $\sD\sH_E$. This gives rise to an interesting action of a central extension of the modular group $\mathsf{SL}_2(\ZZ)$ on the elliptic Hall algebra $\sD\sH_E$.

Based on the aforementioned results, Schiffmann stated in \cite{OlivierNotes} a conjecture asserting that for any  $\kk$-linear $\Ext$-finite hereditary abelian categories $\kA$ and $\kB$, 
any ``friendly'' exact equivalence $F: \kD^b(\kA) \lar \kD^b(\kB)$ induces an isomorphism of double Hall algebras $F_\ast: \sD\sH_\kA \lar \sD\sH_\kB$ and suggested an explicit formula for $F_\ast$. In \cite{Cramer_DoubleHallAlgebras}, Cramer provided a solution to Schiffmann's conjecture under certain additional hypotheses on $\kA$ and $\kB$.

\begin{theorem*}[{\cite[Theorem 1]{Cramer_DoubleHallAlgebras}} and Theorem \ref{thm:CramerIsomorphism}]
Let $F: \kD^b(\kA)\longrightarrow\kD^b(\kB)$ be an equivalence of triangulated categories. Then, under appropriate additional hypotheses on $\kA, \kB$ and $F$, it  induces an algebra isomorphism between the corresponding reduced Drinfeld doubles, defined by
\begin{align}
\label{eq:CramerIsomorphism}
\nonumber
F_*:\qquad \sD\sH_\kA & \longrightarrow \sD\sH_\kB \\
[X]^\pm \cdot k_\alpha & \longmapsto \langle X,X \rangle^n [B_X]^{\pm\varepsilon(n)}k_{\varepsilon(n)n\cdot B_X}\cdot k_{F(\alpha)}
\end{align}
for $X\in\kA$ and $B_X\in\kB$ with $F(X)\cong B_X[n]$ and $\alpha\in\sK_0(\kA)$. Here, $\varepsilon(n)=(-1)^n$ and $\langle X,X\rangle = \sqrt{\dfrac{\big|\Hom_{\kA}(X, X)\big|}{\big|\Ext^1_{\kA}(X, X)\big|}}$.
\end{theorem*}

It turns out, however, that Cramer's proof contains a gap. His argument essentially relies on the following statement from a work of Kapranov \cite[Proposition 2.4.3]{Kapranov_HeisenbergDouble}.

\smallskip
\noindent
\textit{Statement}. For any objects $X, Y, Z \in \kD = \kD^b(\kA)$, let
\begin{align*}
c_{X,Y}^{Z}
&= \Big|\bigl\{f \in \Hom_{\kD}(X, Y)~\big\vert~\cone(f)\cong Z\bigr\}\Big|
\end{align*}
and
\begin{align*}
T_{X,Y}^{Z}
&= \Big|\bigl\{(f, g, h) ~\big\vert~ X \stackrel{f}\lar Y \stackrel{g}\lar Z \stackrel{h}\lar X[1] \; \text{is distinguished}\bigr\}\Big|.
\end{align*}
Then for any objects $L,X,Y,N\in \kA$ we have:
\begin{equation}\label{E:KaFormula}
T_{X,Y}^{L[1] \oplus N}=c_{X,Y}^{L[1] \oplus N} \big|\Aut_\kA(L)\big| \big|\Aut_\kA(N)\big|\big|\Ext^1_{\kA}(N, L)\big|.
\end{equation}
Unfortunately, the formula (\ref{E:KaFormula}) is in general incorrect, which invalidates the proof of \cite[Theorem 1]{Cramer_DoubleHallAlgebras}. For the same reason, the proof of \cite[Theorem 3.9]{BurbanSchiffmann_HallAlgebraEllipticCurve} is also incorrect.

Fortunately, Cramer's proof can be repaired in such a way that the overall strategy of his work \cite{Cramer_DoubleHallAlgebras} is preserved. Since Cramer's result plays a fundamental role in the theory of Hall algebras and the corresponding isomorphism \eqref{eq:CramerIsomorphism} has been used crucially in several subsequent works (see e.g.~\cite{BSDrinfeldBeck, BSWeightLine}), we believe that it is appropriate to provide a detailed corrected proof.  
\\[+4pt]
\emph{Acknowledgement.} The first-named author is thankful to Francesco Sala and Federico Volpe for drawing his attention to 
a problematic place in the proof of  \cite[Proposition 2.4.3]{Kapranov_HeisenbergDouble}.
This work was funded by the German Research Foundation SFB-TRR 358/1 2023 – 491392403.

\section{Hall algebras and Drinfeld doubles}
We consider hereditary abelian categories $\kA$ over a finite field $\kk=\FF_q$ such that $\Hom_\kA(X,Y)$ and $\Ext_\kA^1(X,Y)$ are finite-dimensional for all $X,Y\in \kA$. Given objects $X,Y,Z\in\kA$, we write
\begin{align*}
P_{X,Y}^Z&=\left|\left\{(f,g)\in\Hom_\kA(Y,Z)\times \Hom_\kA(Z,X)~\bigg\vert~
0\rightarrow Y\stackrel{f}{\rightarrow} Z \stackrel{g}{\rightarrow} X \rightarrow 0 \text{ is exact}\right\}\right|,\\
a_X &= \big|\Aut_{\kA}(X)\big|
\end{align*}
and
\begin{align*}
\langle X, Y\rangle &= v^{\dim_\kk\Hom_\kA(X,Y)-\dim_\kk\Ext_\kA^1(X,Y)}, \\
(X,Y) &= \langle X, Y\rangle\langle Y, X\rangle
\end{align*}
with $v=\sqrt{q}$. Note that $\langle X,Y\rangle$ only depends on the corresponding classes $\overline{X}, \overline{Y}\in\sK_0(\kA)$. Hence we can view $\langle -,-\rangle$ and $(-,-)$ as functions on $\sK_0(\kA)\times \sK_0(\kA)$.

\begin{definition}[\cite{Ringel_HallAlgebrasAndQuantumGroups}]
The \emph{extended twisted Hall algebra} $\sH_\kA$ of the abelian category $\kA$ is an associative $\widetilde\QQ$-algebra with underlying vector space
\begin{align*}
\sH_\kA=\overline{\sH}_\kA\otimes_{\widetilde\QQ}\widetilde\QQ\left[\sK_0(\kA)\right]
\quad\text{ with }\quad
\overline{\sH}_\kA=\bigoplus_{[Z]\in\Ob(\kA)_{/\cong}}\widetilde\QQ[Z]
\end{align*}
where $\widetilde\QQ\left[\sK_0(\kA)\right]$ denotes the group algebra of the Grothendieck group $\sK_0(\kA)$. We use the notation $[X]k_\alpha=[X]\otimes k_\alpha$ for elements in $\sH_\kA$. The multiplication is given by
\begin{align*}
[X][Y] &= \langle X,Y \rangle \sum_{[Z]} \frac{P_{X,Y}^Z}{a_Xa_Y} [Z]
\qquad\text{ and }\qquad
k_\alpha[X] = (\alpha, X)[X]k_\alpha
\end{align*}
for all $X,Y\in\kA$ and $\alpha\in\sK_0(\kA)$ where the sum runs over all isomorphism classes of objects in $\kA$.
It has been shown in \cite{Ringel_HallAlgebrasAndQuantumGroups} that this turns $\sH_\kA$ into an associative algebra with unit $1=[0]\otimes k_0$.
\end{definition}

From now on we restrict ourselves to the following two cases:
\begin{enumerate}
\item[(1)] $\kA$ is a length category, or
\item[(2)] $\kA=\mathit{Coh}(\XX)$ the category of coherent sheaves on a  complete non-commutative hereditary curve $\XX$ over $\kk$; see e.g.~\cite{BurbanDrozd_NonCommutativeSchemes} for the corresponding definitions. 
\end{enumerate}

Following \cite{Green_HallAlgebras}, the algebra $\sH_\kA$ can be equipped with a comultiplication $\Delta$ and a bialgebra pairing $\varphi:\sH_\kA\times\sH_\kA\longrightarrow \widetilde\QQ$, defined by
\begin{align*}
\Delta\bigl([Z]k_\alpha\bigr) &= \sum_{[X],[Y]} \langle X,Y\rangle \frac{P_{X,Y}^Z}{a_Z} [X]k_{Y+\alpha}\otimes [Y]k_\alpha, \\
\varphi\bigl([X]k_\alpha,[Y]k_\beta\bigr) &= (\alpha, \beta) \frac{\delta_{X,Y}}{a_X}.
\end{align*}
In the case of a length category this turns $\sH_\kA$ into a Hopf algebra \cite{Xiao}. Otherwise the coproduct might be an infinite sum. In this case, one can introduce the completed tensor product and obtains a so-called \emph{topological bialgebra}, see \cite[Appendix B]{BurbanSchiffmann_HallAlgebraEllipticCurve} for further details.

\begin{definition}
Let $\sH_\kA^+$ and $\sH_\kA^-$ be two copies of the algebra $\sH_\kA$, whose elements we temporarily denote by $[X]^\pm k_\alpha^\pm$. The \emph{Drinfeld double} $\widetilde\sD\sH_\kA$ of $\sH_\kA$ with respect to the pairing $\varphi$ is defined as the free product of algebras $\sH_\kA^+$ and $\sH_\kA^-$ modulo the relation
\begin{align}\label{eq:DrinfeldDoubleRelationGeneral}
\tag{R(a,b)}
\sum_{i,j}\varphi\left(a^{(2)}_i,b^{(1)}_j\right)a^{(1)+}_i b^{(2)-}_j
&= \sum_{i,j}\varphi\left(a^{(1)}_i,b^{(2)}_j\right) b^{(1)-}_j a^{(2)+}_i
\end{align}
for all $a,b\in\sH_\kA$ where we use the notation
\begin{align*}
\Delta(a) = \sum_{i}a^{(1)}_i\otimes a^{(2)}_i
\qquad\text{ and }\qquad
\Delta(b) = \sum_{j}b^{(1)}_j\otimes b^{(2)}_j.
\end{align*}
\end{definition}

The following result is well-known in the case where $\kA$ is a length category and hence $\sH_\kA$ is a Hopf algebra \cite[Section 3.2]{Joseph_QuantumGroups};  otherwise we refer to \cite[Appendix B]{BurbanSchiffmann_HallAlgebraEllipticCurve}.
\begin{theorem}\label{thm:DrinfeldDoubleIsoTensorproduct}
The multiplication map $\sH_\kA^+\otimes\sH_\kA^-\longrightarrow\widetilde\sD\sH_\kA$ is an isomorphism of $\widetilde\QQ$-vector spaces.
\end{theorem}

We can therefore identify $\widetilde\sD\sH_\kA$ with the tensor product of two copies of $\sH_\kA$ and write from now on
\begin{align*}
[X]^+k_\alpha^+=[X]k_\alpha\otimes 1
\qquad \text{ and }\qquad
[X]^-k_\alpha^-=1\otimes [X]k_\alpha
\end{align*}
for all $X\in\kA$ and $\alpha\in\sK_0(\kA)$.

\begin{definition}
The \emph{reduced Drinfeld double} $\sD\sH_\kA$ is the quotient of $\widetilde\sD\sH_\kA$ by the two-sided ideal generated by all elements $(k_\alpha\otimes 1)-(1\otimes k_{-\alpha})$ for $\alpha\in\sK_0(\kA)$.
\end{definition}

The goal for the remainder of this section is to fully express the algebra $\sD\sH_\kA$ in terms of concrete relations. In order to rewrite these relations in a form that is best suited to our purposes, we introduce the following notation.

\begin{definition}
For objects $L,X,Y,N\in \kA$ we write
\begin{align*}
c_{X,Y}^{L[1]\oplus N} 
&= \big|\bigl\{\psi\in \Hom_{\kD^b(\kA)}(X,Y)~\vert~\cone(\psi)\cong L[1]\oplus N\bigr\}\big|.
\end{align*}
\end{definition}

We proceed by proving two properties of these numbers which relate them to classical Hall numbers.

\begin{lemma}\label{lemma:cEqualP}
Let $X,Y,Z$ be objects in $\kA$ such that $\Hom_\kA(X,Y)=0$. Then
\begin{align*}
c_{X,Y[1]}^{Z[1]} = \frac{P_{X,Y}^Z}{a_Z}.
\end{align*}
\end{lemma}
\begin{proof}
Since $\Hom_{\kD^b(\kA)}(X,Y[1])\cong\Ext^1_\kA(X,Y)$, the number $c_{X,Y[1]}^{Z[1]}$ counts equivalence classes of short exact sequences $0\longrightarrow Y\longrightarrow Z\longrightarrow X\longrightarrow 0$. Recall that two short exact sequences $(f,g)$ and $(f',g')$ are equivalent if and only if there exists a commutative diagram
\begin{align*}
\xymatrix{
0 \ar[r] & Y \ar[r]^f \ar@{=}[d] & Z \ar[r]^g \ar[d] & X \ar[r] \ar@{=}[d] & 0\\
0 \ar[r] & Y \ar[r]^{f'}        & Z \ar[r]^{g'}        & X \ar[r]        & 0.
}
\end{align*}
The claim of the lemma will follow once we know that every equivalence class has precisely $a_Z$ elements. It is clear that the automorphism group $\Aut_\kA(Z)$ acts transitively on every equivalence class. It remains to show that this action is free, provided $\Hom_\kA(X,Y)=0$. If $\sigma\in\Aut_\kA(Z)$ is an automorphism that stabilises a short exact sequence $(f,g)$, then we obtain commutative diagrams
\begin{align*}
\xymatrix{
0 \ar[r] & Y \ar[r]^f \ar[d]_\id & Z \ar[r]^g \ar[d]^{\sigma} & X \ar[r] \ar[d]^\id  & 0\\
0 \ar[r] & Y \ar[r]_f        & Z \ar[r]_g        & X \ar[r]        & 0
}
\hspace{2cm}
\xymatrix{
0 \ar[r] & Y \ar[r]^f \ar[d]_0 & Z \ar[r]^g \ar[d]_{\tilde{\sigma}} & X \ar[r] \ar[d]^0 \ar@{.>}[dl]|{\exists\psi} & 0\\
0 \ar[r] & Y \ar[r]_f        & Z \ar[r]_g        & X \ar[r]        & 0
}
\end{align*}
with $\tilde{\sigma}=\id-\sigma$. By the universal property of the cokernel, there exists a morphism $\psi:X\longrightarrow Z$ such that $\psi g=\tilde{\sigma}$. In particular, we get
\begin{align*}
g \psi g = g\tilde{\sigma} = 0 \qquad\Rightarrow\qquad g\psi=0
\end{align*}
because $g$ is an epimorphism. By the universal property of the kernel, the morphism $\psi$ factors through $Y$. From the assumption $\Hom_\kA(X,Y)=0$ we can deduce that $\psi=0$ and hence $\tilde{\sigma}=0$ and $\sigma=\id$. This shows that the $\Aut_\kA(Z)$-action on every equivalence class of short exact sequences $0\longrightarrow Y\longrightarrow Z\longrightarrow X\longrightarrow 0$ is free. Therefore, every equivalence class has precisely $a_Z$ elements which finishes the proof.
\end{proof}

\begin{lemma}\label{lemma:cEqualSum}
Let $L,X,Y,N$ be objects in $\kA$. Then
\begin{align*}
c_{X,Y}^{L[1]\oplus N} = \frac{1}{a_La_N}\sum_{[M]}\frac{P_{N,M}^YP_{M,L}^X}{a_M}.
\end{align*}
\end{lemma}
\begin{proof}
Since $\kA$ is hereditary, we have
\begin{align*}
\cone(\psi)\cong L[1]\oplus N 
\qquad\Longleftrightarrow\qquad
\ker(\psi)\cong L\text{ and }\cok(\psi)\cong N
\end{align*}
for every morphism $\psi\in\Hom_\kA(X,Y)$. For every object $M\in\kA$, the map
\begin{align*}
\left\{ 
\begin{array}{c}
\text{pairs of short exact sequences}\\[+3pt]
0\longrightarrow L\longrightarrow X \longrightarrow M\longrightarrow 0\\
0\longrightarrow M\longrightarrow Y \longrightarrow N\longrightarrow 0
\end{array}\right\}
&\longrightarrow
\left\{ 
\begin{array}{c}
\text{exact sequences}\\
0\longrightarrow L\longrightarrow X \stackrel{\psi}{\longrightarrow} Y\longrightarrow N\longrightarrow 0\\
\text{with } \mathsf{Im}(\psi)\cong M
\end{array}\right\}
\end{align*}
is surjective and the fibre over every element of the target has precisely $a_M$ elements. Taking cardinalities and summing over all isomorphism classes of objects $M\in\kA$, we obtain
\begin{align*}
\sum_{[M]}\frac{P_{N,M}^YP_{M,L}^X}{a_M}
&=\left|
\left\{ 
\begin{array}{c}
\text{exact sequences}\\
0\longrightarrow L\longrightarrow X \stackrel{\psi}{\longrightarrow} Y\longrightarrow N\longrightarrow 0
\end{array}\right\}
\right|
\\
&=a_La_N\cdot
\left|
\left\{
X \stackrel{\psi}{\longrightarrow} Y
\right|\left.
\begin{array}{c}
\ker(\psi)\cong L, \\
\cok(\psi)\cong N
\end{array}\right\}
\right|
=a_La_N\cdot c_{X,Y}^{L[1]\oplus N}.
\qedhere
\end{align*}
\end{proof}

Now we are ready to give a complete list of relations for the reduced Drinfeld double of $\kA$.

\begin{proposition}\label{P:RelationsDouble}
The reduced Drinfeld double $\sD\sH_\kA$ is isomorphic to the free associative $\widetilde\QQ$-algebra on the vector space $\sH_\kA\otimes\sH_\kA$ modulo the relations
\begin{align}
\label{eq:relationLeftRight}
([X] \otimes 1)(1 \otimes [Y]) &= ([X] \otimes [Y])  \\[+8pt]
\label{eq:relationK0right}
(1 \otimes k_X)(1 \otimes k_Y) &= (1 \otimes k_{X+Y})  \\
\label{eq:relationK0left}
(k_X \otimes 1)(k_Y \otimes 1) &= (k_{X+Y} \otimes 1)  \\
\label{eq:relationK0reduced}
k_X \otimes 1 &= 1 \otimes k_{-X} \\[+8pt]
\label{eq:relationK0swapRight}
(1 \otimes k_X)(1 \otimes [Y]) = (1 \otimes k_X[Y]) &= (X,Y)(1\otimes [Y]k_X) = (X, Y) (1 \otimes [Y])(1 \otimes k_X)  \\
\label{eq:relationK0swapLeft}
(k_X \otimes 1)([Y] \otimes 1) = (k_X [Y] \otimes 1) &= (X,Y)([Y]k_X\otimes 1) =  (X, Y) ([Y] \otimes 1)(k_X \otimes 1)  \\[+8pt]
\label{eq:relationHallProductRight}
(1 \otimes [X])(1 \otimes [Y]) &= \langle X, Y \rangle \sum_{[Z]} \frac{P_{X,Y}^Z}{a_Xa_Y} (1 \otimes [Z]) \\
\label{eq:relationHallProductLeft}
([X] \otimes 1)([Y] \otimes 1) &= \langle X, Y \rangle \sum_{[Z]} \frac{P_{X,Y}^Z}{a_Xa_Y} ([Z] \otimes 1)  
\end{align}
\vspace{-8pt}
\begin{align}\label{eq:relationDrinfeld}
\nonumber
&\frac{\langle X, X \rangle}{\langle Y, X \rangle}\sum_{[L], [N]} \frac{\langle N, L \rangle}{\langle L, L \rangle} \frac{a_L a_N}{a_Xa_Y} c_{Y, X}^{N[1] \oplus L}  ([L] \otimes [N]k_{L-X}) \nonumber \\
&\hspace{2.5cm} =\frac{\langle Y, Y \rangle}{\langle X, Y \rangle}\sum_{[L], [N]} \frac{\langle L, N \rangle}{\langle N, N \rangle} \frac{a_L a_N}{a_Xa_Y} c_{X, Y}^{L[1] \oplus N} (1 \otimes [N])([L] \otimes k_{X-L}). 
\end{align}
for all $X,Y \in \kA$.
\end{proposition}
\begin{proof}
The relation \eqref{eq:relationLeftRight} is a consequence of our identification of $\widetilde\sD\sH_\kA$ with $\sH_\kA^+\otimes\sH_\kA^-$ according to Theorem \ref{thm:DrinfeldDoubleIsoTensorproduct}. The relations \eqref{eq:relationK0right}, \eqref{eq:relationK0left}, \eqref{eq:relationHallProductRight} and \eqref{eq:relationHallProductLeft} already hold in $\sH_\kA^\pm$ while \eqref{eq:relationK0reduced} follows from the definition of the \emph{reduced} Drinfeld double. 

\noindent
It remains to check that \eqref{eq:relationK0swapRight},\eqref{eq:relationK0swapLeft} and \eqref{eq:relationDrinfeld} are together equivalent to \eqref{eq:DrinfeldDoubleRelationGeneral} for all $a,b\in\sH_\kA$.\\
Given $Z\in\kA$ and $\alpha\in\sK_0(\kA)$, the relation $\mathrm{R}(k_\alpha, [Z])$ reads
\begin{align*}
\sum_{[X],[Y]} \langle X,Y\rangle \frac{P_{X,Y}^Z}{a_Z}\!\underbrace{\varphi(k_\alpha, [X]k_Y)}
_{\scriptsize =\begin{cases}(\alpha,Z),& \text{if}X=0\\ 0,& \text{else} \end{cases}}\! (k_\alpha\otimes [Y])
&=\sum_{[X],[Y]} \langle X,Y\rangle \frac{P_{X,Y}^Z}{a_Z}\!\underbrace{\varphi(k_\alpha, [Y])}
_{\scriptsize =\begin{cases}1,& \text{if }Y=0\\ 0,& \text{else} \end{cases}}\! (1\otimes [X]k_Y)(k_\alpha \otimes 1)\\[+5pt]
\Longleftrightarrow\qquad
(\alpha,Z)(k_\alpha\otimes [Z]) &= (1\otimes [Z])(k_\alpha\otimes 1)\\
\Longleftrightarrow\quad\qquad\;\:
(1 \otimes k_{-\alpha}[Z]) &= (-\alpha,Z)(1\otimes [Z]k_{-\alpha})
\end{align*}
which shows \eqref{eq:relationK0swapRight}. In a similar fashion, \eqref{eq:relationK0swapLeft} is implied by relations of the form $R([Z],k_\alpha)$.\\
For the last step we pick $a=[X]k_\alpha$ and $b=[Y]k_\beta$ in $\sH_\kA$.
The relevant coproducts are
\begin{align*}
\Delta([X]k_\alpha) &= \sum_{[X'],[X'']} \langle X',X''\rangle \frac{P_{X',X''}^X}{a_X} [X']k_{X''+\alpha}\otimes [X'']k_\alpha,\\
\Delta([Y]k_\beta) &= \sum_{[Y'],[Y'']} \langle Y',Y''\rangle \frac{P_{Y',Y''}^Y}{a_Y} [Y']k_{Y''+\beta} \otimes [Y'']k_\beta.
\end{align*}
The relation \eqref{eq:DrinfeldDoubleRelationGeneral} reads
\begin{align*}
&\sum_{} \langle X',X''\rangle\langle Y',Y''\rangle \frac{P_{X',X''}^XP_{Y',Y''}^Y}{a_Xa_Y}
\overbrace{\varphi([X'']k_\alpha,[Y']k_{Y''+\beta})}^{=(\alpha,Y''+\beta)\frac{\delta_{X'',Y'}}{a_{X''}}}\cdot
[X']k_{X''+\alpha}\otimes[Y'']k_\beta
\\ =
&\sum_{} \langle X',X''\rangle\langle Y',Y''\rangle \frac{P_{X',X''}^XP_{Y',Y''}^Y}{a_Xa_Y}
\underbrace{\varphi([X']k_{X''+\alpha},[Y'']k_\beta)}_{=(X''+\alpha,\beta)\frac{\delta_{X',Y''}}{a_{X'}}}\cdot
(1\otimes [Y']k_{Y''+\beta})([X'']k_\alpha \otimes 1)
\end{align*}
where both sums run over all isomorphism classes of objects $X',X'',Y',Y''\in\kA$. On the left-hand side we relabel them as $L=X',M=X''=Y',N=Y''$ and on the right-hand side as $L=X'',M=X'=Y'',N=Y'$.  We obtain
\begin{align*}
&\sum_{} \langle L,M\rangle\langle M,N\rangle \frac{P_{L,M}^XP_{M,N}^Y}{a_Xa_Ya_M}
(\alpha,N+\beta)\cdot
[L]k_{M+\alpha}\otimes[N]k_\beta
\\ 
&\hspace{2cm}=\sum_{} \langle M,L\rangle\langle N,M\rangle \frac{P_{M,L}^XP_{N,M}^Y}{a_Xa_Ya_M}
(L+\alpha,\beta)\cdot
(1\otimes [N]k_{M+\beta})([L]k_\alpha \otimes 1)
\end{align*}
where the sums now run over all isomorphism classes of objects $L,M,N\in\kA$. Using \eqref{eq:relationK0reduced}--\eqref{eq:relationK0swapLeft} we get
\begin{align*}
(\alpha,N+\beta) [L]k_{M+\alpha}\otimes[N]k_\beta
&= (\alpha,\beta)(\alpha,N) [L]k_{M}\otimes k_{-\alpha}[N]k_\beta \\
=(\alpha,\beta) [L]k_{M}\otimes [N]k_{-\alpha}k_\beta
&=(\alpha,\beta) [L]k_{M}\otimes [N]k_{\beta-\alpha}
\end{align*}
and
\begin{align*}
(L+\alpha,\beta) (1\otimes [N]k_{M+\beta})([L]k_\alpha \otimes 1)
&= (\alpha,\beta)(L,\beta) (1\otimes [N]k_{M})(k_{-\beta}[L]k_\alpha \otimes 1) \\
= (\alpha,\beta) (1\otimes [N]k_{M})([L]k_{-\beta}k_\alpha \otimes 1)
&= (\alpha,\beta) (1\otimes [N]k_{M})([L] \otimes k_{\beta-\alpha})
\end{align*}
such that we can cancel factors $(\alpha,\beta)$ and $1\otimes k_{\beta-\alpha}$ (from the left). In particular, relation $\mathrm{R}([X]k_\alpha,[Y]k_\beta)$ no longer depends on $\alpha$ and $\beta$ and simplifies to
\begin{align*}
&\sum_{} \langle L,M\rangle\langle M,N\rangle \frac{P_{L,M}^XP_{M,N}^Y}{a_Xa_Ya_M}
([L]k_M\otimes[N])
\\ 
&\hspace{2cm}=\sum_{} \langle M,L\rangle\langle N,M\rangle \frac{P_{M,L}^XP_{N,M}^Y}{a_Xa_Ya_M}
(1\otimes [N]k_M)([L] \otimes 1)\\[+5pt]
\Longleftrightarrow\qquad
&\sum_{} \frac{\langle L,M\rangle}{\langle N,M\rangle} \frac{P_{L,M}^XP_{M,N}^Y}{a_Xa_Ya_M}
([L]\otimes[N]k_M^{-1})
=
\sum_{} \frac{\langle N,M\rangle}{\langle L,M\rangle} \frac{P_{M,L}^XP_{N,M}^Y}{a_Xa_Ya_M}
(1\otimes [N])([L] \otimes k_M).
\end{align*}
In order to eliminate the summation over $M$, we compute
\begin{align*}
\frac{\langle L,M\rangle}{\langle N,M\rangle}
= \frac{\langle L,X-L\rangle}{\langle N,X-L\rangle}
=\langle L-N,X\rangle\frac{\langle N,L\rangle}{\langle L,L\rangle}
=\langle X-Y,X\rangle\frac{\langle N,L\rangle}{\langle L,L\rangle}
=\frac{\langle X,X\rangle}{\langle Y,X\rangle}\frac{\langle N,L\rangle}{\langle L,L\rangle}
\end{align*}
and
\begin{align*}
\frac{\langle N,M\rangle}{\langle L,M\rangle}
= \frac{\langle N,Y-N\rangle}{\langle L,Y-N\rangle}
=\langle N-L,Y\rangle\frac{\langle L,N\rangle}{\langle N,N\rangle}
=\langle Y-X,Y\rangle\frac{\langle L,N\rangle}{\langle N,N\rangle}
=\frac{\langle Y,Y\rangle}{\langle X,Y\rangle}\frac{\langle L,N\rangle}{\langle N,N\rangle}
\end{align*}
which leads to
\begin{align}\label{eq:relationDrinfeldWithM}
\nonumber
&\frac{\langle X,X\rangle}{\langle Y,X\rangle}\sum_{[L],[N]} \frac{\langle N,L\rangle}{\langle L,L\rangle} \frac{1}{a_Xa_Y}\overbrace{\sum_{[M]}\frac{P_{L,M}^XP_{M,N}^Y}{a_M}}^{=a_La_Nc_{Y,X}^{N[1]\oplus L}}
([L]\otimes[N]k_{L-X})
\\
&\hspace{2cm}=\frac{\langle Y,Y\rangle}{\langle X,Y\rangle}\sum_{[L],[N]} \frac{\langle L,N\rangle}{\langle N,N\rangle} \frac{1}{a_Xa_Y}\underbrace{\sum_{[M]}\frac{P_{M,L}^XP_{N,M}^Y}{a_M}}_{=a_La_Nc_{X,Y}^{L[1]\oplus N}}
(1\otimes [N])([L] \otimes k_{X-L}).
\end{align}
By Lemma \ref{lemma:cEqualSum}, this is equivalent to \eqref{eq:relationDrinfeld}, which completes the proof that \eqref{eq:relationLeftRight}--\eqref{eq:relationDrinfeld} are satisfied. It is easy to see that this list of relations is complete.
\end{proof}

\begin{remark} The list of relations given in Proposition \ref{P:RelationsDouble} is, of course, not new. It already appears, essentially in the same form, in the work of Sevenhant and Van den Bergh \cite[Section 8]{SevenhBergh}.
We point out that the relations \eqref{eq:relationLeftRight}--\eqref{eq:relationHallProductLeft} coincide with \cite[(5.2)--(5.9)]{Cramer_DoubleHallAlgebras} modulo the fact that our conventions for the multiplication in the twisted Hall algebra differ from those used by Cramer. In the same way, relation \eqref{eq:relationDrinfeldWithM} still coincides with \cite[(4.5/4.6)]{Cramer_DoubleHallAlgebras}. While we use Lemma \ref{lemma:cEqualSum} to further simplify this, Cramer's computations are based on \cite[Proposition 2.4.3]{Kapranov_HeisenbergDouble} which is in general incorrect (see Example \ref{ex:counterexample} below).
More precisely, Cramer substituted
\begin{align*}
\sum_{[M]}\frac{P_{M,L}^XP_{N,M}^Y}{a_M} = \frac{T_{X,Y}^{L[1]\oplus N}}{\big|\Ext^1_\kA(N,L)\big|}
\end{align*}
on the right-hand side of \eqref{eq:relationDrinfeldWithM} where
\begin{align*}
T_{X,Y}^{L[1]\oplus N}=
\left|\left\{
\begin{array}{c}
\text{distinguished triangles}\\
X\longrightarrow Y \longrightarrow L[1]\oplus N \longrightarrow X[1] 
\end{array}
\right\}\right|
\end{align*}
(and analogously on the left-hand side of \eqref{eq:relationDrinfeldWithM}).
Hence, Cramer's relation \cite[(5.10/5.11)]{Cramer_DoubleHallAlgebras} is incorrect and parts of the proof of his main result need to be modified in accordance with the corrected relation \eqref{eq:relationDrinfeld}.
\end{remark}

\begin{example}\label{ex:counterexample}
A concrete counterexample to \cite[Proposition 2.4.3]{Kapranov_HeisenbergDouble} (respectively formula \eqref{E:KaFormula}) can be constructed for the hereditary category $\kA=\kk[\, t\,]\!-\!\textsf{fdmod}$. We consider the objects
\begin{align*}
X=Y=\kk[\,t\,]/(t^2) \qquad\text{ and }\qquad L=N=\kk[\,t\,]/(t)
\end{align*}
such that $\Ext_\kA^1(N,L)\cong\kk$ and the right-hand side of \eqref{E:KaFormula} equals $q(q-1)^3$. In order to count the number of distinguished triangles, let us fix a map $\psi: X\longrightarrow Y$ with $\cone(\psi)\cong L[1]\oplus N$. Since $\kA$ is hereditary, $\psi$ has kernel $L$ and cokernel $N$ and we get an exact sequence
\begin{align*}
0\longrightarrow L \stackrel{\iota}{\longrightarrow} X \stackrel{\psi}{\longrightarrow} Y \stackrel{\pi}{\longrightarrow} N \longrightarrow 0.
\end{align*}
The axioms of triangulated categories imply that the automorphism group
\begin{align*}
G_{LN} = \Aut_{D^b(\kA)}\bigl(L[1]\oplus N\bigr)\cong\begin{pmatrix} \Aut_\kA(L) & \Ext_\kA^1(N,L) \\ 0 & \Aut_\kA(N) \end{pmatrix}
\end{align*}
 acts transitively on the set of distinguished triangles of the form 
$X \stackrel{\psi}\lar Y \lar L[1] \oplus N \lar X[1]$.
 If this action were always free, we could deduce \eqref{E:KaFormula} from it. However, we will show that this is not the case in our example. Given an element $\begin{pmatrix}f & w\\0 & g\end{pmatrix} \in G_{LN}$
  belonging to  the stabiliser of an appropriate  distinguished triangle, we obtain a commutative diagram

\begin{align*}
\xymatrix{
X \ar[rr]^\psi \ar@{=}[dd] && Y \ar[rr]^(.4){\begin{pmatrix}\xi\\[-3pt]\pi\end{pmatrix}} \ar@{=}[dd] && L[1]\oplus N \ar[rr]^(.55){\begin{pmatrix}\iota[1]\;\eta\end{pmatrix}} \ar[dd]^(.4){\begin{pmatrix}f & w\\0 & g\end{pmatrix}} && X[1] \ar@{=}[dd] \\
&&&&&& \\
X \ar[rr]^\psi        && Y \ar[rr]^(.4){\begin{pmatrix}\xi\\[-3pt]\pi\end{pmatrix}}        && L[1]\oplus N \ar[rr]^(.55){\begin{pmatrix}\iota[1]\;\eta\end{pmatrix}}        && X[1].       
}
\end{align*}
Since $\pi$ is a cokernel of $\psi$, it is an epimorphism in $\kA$ which yields $g=\id$. Analogously, we obtain $f=\id$. It follows that the stabiliser group of the above triangle is isomorphic to the abelian group
\begin{align*}
\left\{
w\in\Ext_\kA^1(N,L)\;\bigg\vert\; \begin{array}{c} w\circ\pi=0 \\ \iota\circ w=0 \end{array}
\right\}
=
\ker\left(
\Ext_\kA^1(N,L)\xrightarrow{~~\begin{pmatrix}\pi^*\\[-3pt] \iota_* \end{pmatrix}~~} \Ext_\kA^1(Y,L) \oplus \Ext_\kA^1(N,X)
\right).
\end{align*}
A direct computation shows that the maps $\pi^*$ and $\iota_*$ are both zero. This also can be deduced from the fact that
$$
0 \lar \kk[\,t\,]/(t) \lar \kk[\,t\,]/(t^2) \lar \kk[\,t\,]/(t) \lar 0
$$
is an almost split sequence in $\kA$.
Hence, the aforementioned stabiliser group is isomorphic to $\Ext_\kA^1(N,L)$. We can conclude that the desired number of distinguished triangles equals
\begin{align*}
T_{X,Y}^{L[1]\oplus N}= c_{X,Y}^{L[1]\oplus N}\cdot \frac{\big|\Aut_{D^b(\kA)}(L[1]\oplus N)\big|}{|\Ext_\kA^1(N,L)|}
= c_{X,Y}^{L[1]\oplus N}\cdot \big|\Aut_{\kA}(L)\big| \big|\Aut_{\kA}(N)\big| = (q-1)^3.
\end{align*} 
Since the right-hand side of \eqref{E:KaFormula} equals $q(q-1)^3$, this shows that \cite[Proposition 2.4.3]{Kapranov_HeisenbergDouble} does not hold in general. The reason for this failure is that the above map 
$\begin{pmatrix}\pi^*\\[-3pt] \iota_* \end{pmatrix}$ 
need not be injective in general.
\end{example}

\begin{remark}\label{remark:SimplifiedRelationHomVanishing}
Note that \eqref{eq:relationDrinfeld} simplifies in the case where there are no non-trivial morphisms between the objects $X$ and $Y$ in one or even both directions. For instance, if $\Hom_{\kA}(X,Y)=0$ then we have
\begin{align}
\label{eq:cWithHomVanishing}
c_{X,Y}^{L[1]\oplus N} =
\begin{cases}
1, & \text{ if } L=X \text{ and } N=Y,\\
0, & \text{ otherwise}
\end{cases}
\end{align}
such that there is only one non-zero summand in the sum on the right-hand side. In particular, for $X,Y\in\kA$ with $\Hom_\kA(X,Y)=0=\Hom_\kA(Y,X)$ the relation \eqref{eq:relationDrinfeld} simply reads
\begin{align*}
([X]\otimes [Y])=(1\otimes [Y])([X]\otimes 1).
\end{align*}
\end{remark}

\section{Derived equivalences and reduction of relations}
Let $\kA,\kB$ be hereditary categories and let $F:~\kD^b(\kA)\longrightarrow\kD^b(\kB)$ be an equivalence of triangulated categories. For every $i\in\ZZ$ we define
\begin{align*}
\kA_i&=\bigl\{X\in \kA~\vert~F(X)\in \kB[\,i\,]\bigr\}\subseteq\kA\\
\kB_i&=\bigl\{F(X)[-i]~\vert~X\in \kA_i\bigr\}\subseteq\kB.
\end{align*} 
It is clear that
\begin{align}
\label{eq:orthogonal}
\begin{array}{rl}
\Hom_\kA(\kA_i,\kA_j) =0,&\qquad\text{ unless }\quad j\in\{i,i+1\},\\[+3pt]
\Ext^1_\kA(\kA_i,\kA_j) =0,&\qquad\text{ unless }\quad j\in\{i-1,i\},\\[+3pt]
\Hom_\kB(\kB_i,\kB_j) =0,&\qquad\text{ unless }\quad j\in\{i-1,i\},\\[+3pt]
\Ext^1_\kB(\kB_i,\kB_j) =0,&\qquad\text{ unless }\quad j\in\{i,i+1\}.
\end{array}
\end{align}

Since $\kA$ and $\kB$ are hereditary, every indecomposable object in $\kD^b(\kA)$ and $\kD^b(\kB)$ is a complex concentrated in a single degree. Therefore, we get a decomposition
\begin{align*}
\kA = \qquad\ldots\;\vee\; \kA_{-1} \;\vee\; \kA_0 \;\vee\; \kA_1 \;\vee\; \ldots
\end{align*}
such that every object $X\in\kA$ can be written uniquely as a (finite) direct sum
\begin{align*}
X = X_i\oplus X_{i+1} \oplus \ldots \oplus X_{i+n}
\end{align*}
for some $i\in\ZZ$, $n\geq 0$ and $X_j\in\kA_j$. As an immediate consequence of \eqref{eq:orthogonal} we observe that
\begin{align}\label{eq:multiplicativeDecomposition}
[X_i][X_{i+1}]\cdot \ldots \cdot [X_{i+n}] = \prod_{i\leq l<m\leq i+n} \langle X_l,X_m\rangle \cdot [X]
\end{align}
in $\sH_\kA$. Using this, it is straightforward to check that the relations \eqref{eq:relationLeftRight}--\eqref{eq:relationHallProductLeft} are implied for all $X,Y\in\kA$ if we assume that they hold for all $X\in\kA_i$ and $Y\in\kA_j$ for all $i,j\in\ZZ$. In order to get a similar simplification for relation \eqref{eq:relationDrinfeld} we will make use of the following auxiliary statement; see \cite[Lemma 3.3]{BurbanSchiffmann_HallAlgebraEllipticCurve} or \cite[Proposition A.25]{HongTsymbaliuk}.

\begin{lemma}\label{lemma:DrinfeldRelationMultiplicative}
Let $n\geq 1$ and $a, b_0,\ldots,b_n, c_0,\ldots, c_n, d\in\sH_\kA$. Then we have the implications
\begin{align*}
\mathrm{R}\left(a_i^{(k)},b_k\right)\quad \begin{array}{l}\forall\;0\leq k\leq n \\ \forall\; i \end{array} \qquad\Longrightarrow\qquad \mathrm{R}(a,b_0b_1\ldots b_n)
\end{align*}
and
\begin{align*}
\mathrm{R}\left(c_k,d_j^{(k)}\right)\quad \begin{array}{l}\forall\;0\leq k\leq n \\ \forall\; j \end{array} \qquad\Longrightarrow\qquad \mathrm{R}(c_0c_1\ldots c_n,d)
\end{align*}
where we use the notation
\begin{align*}
\Delta^n(a)=\sum_i a_i^{(0)}\otimes\ldots\otimes a_i^{(n)}\qquad\mathrm{and}\qquad
\Delta^n(d)=\sum_j d_j^{(0)}\otimes\ldots\otimes d_j^{(n)}.
\end{align*}
\end{lemma}

\begin{proposition}\label{prop:reductionOfRelations}
Assume that one of the following conditions is satisfied:
\begin{enumerate}
\item $\kA$ is a length category.
\item $\kA=\mathit{Coh}(\XX)$ for a complete non-commutative hereditary curve $\XX$ over $\kk$. Furthermore, $\kA_i=0$ for all $i\neq 0,1$ and $\mathit{Tor}(\XX)\subseteq\kA_1$ \footnote{Note that in this case $(\kA_1,\kA_0)$ is a split torsion pair in $\kA$. Moreover the first part of the condition essentially means that the equivalence $F$ is \emph{friendly} in the sense of \cite[§5.3]{OlivierNotes}}.
\end{enumerate}
Then the relations \eqref{eq:relationLeftRight}--\eqref{eq:relationDrinfeld} for all $X,Y\in \kA$ are implied by the same relations for all $X\in\kA_i$, $Y\in\kA_j$ and $i,j\in\ZZ$.
\end{proposition}
\begin{proof}
We start with the case where $\kA$ is a length category. We will prove using induction on $\ell(X)+\ell(Y)$ that the relation $\mathrm{R}([X],[Y])$ holds whenever one of the two objects $X,Y$ is contained in some $\kA_i$. The initial step is clear since every simple object is contained in one of the subcategories $\kA_i$. Let $r\geq 2$ and suppose the statement holds whenever $\ell(X)+\ell(Y)\leq r$. In order to verify the inductive step, we take two objects $X,Y\in\kA$ with $\ell(X)+\ell(Y)=r+1$. Without loss of generality we can assume that $X\in\kA_i$ for some $i\in\ZZ$. As we have already seen in \eqref{eq:multiplicativeDecomposition}, we can write
\begin{align*}
[Y] = C\cdot [Y_j]\cdot\ldots\cdot[Y_{j+n}]
\end{align*}
for some constant $C\in\widetilde\QQ$ and objects $Y_l\in\kA_l$. On the other hand, we can write
\begin{align*}
\Delta^n([X]) = \sum_k [X_k^{(0)}]\otimes\ldots\otimes [X_k^{(n)}]
\end{align*}
where we omit constants and contributions from $\widetilde\QQ[\sK_0(\kA)]$ since the validity of \eqref{eq:DrinfeldDoubleRelationGeneral} does not depend on them. We claim that for every $0\leq m\leq n$ and for every $k$, the relation $\mathrm{R}(X_k^{(m)},Y_{j+m})$ is satisfied. If $X_k^{(m)}=X\in\kA_i$, then this follows from the assumption. Otherwise, it follows from the induction hypothesis since in this case we have
\begin{align*}
\underbrace{\ell(X_k^{(m)})}_{<\ell(X)}+\underbrace{\ell(Y_{j+m})}_{\leq \ell(Y)} < \ell(X)+\ell(Y).
\end{align*}
Hence, Lemma \ref{lemma:DrinfeldRelationMultiplicative} ensures that $\mathrm{R}([X],[Y])$ holds. This completes the inductive proof that $\mathrm{R}([X],[Y])$ holds whenever one of the objects $X,Y$ is contained in some $\kA_i$. The general case is straightforward using once again the decomposition of the form \eqref{eq:multiplicativeDecomposition} and Lemma \ref{lemma:DrinfeldRelationMultiplicative}. Hence the assertion holds true if $\kA$ is a length category.\\
Now let us assume that $\kA=\mathit{Coh}(\XX)$ and $(\kA_1,\kA_0)$ is a split torsion pair in $\kA$, with all torsion sheaves contained in $\kA_1$. We will use that in this case $\kA$ is equipped with a rank function $\rk$ which is additive on short exact sequences and such that $\mathit{Tor}(\XX)=\bigl\{X\in \kA\vert \rk(X)=0\bigr\}$ is an abelian subcategory.
\\ \underline{Claim 1:} $\mathrm{R}([X],[Y])$ holds whenever $\rk(X)=0$ or $\rk(Y)=0$.\\
Suppose $\rk(Y)=0$. Then we write $X=X_0\oplus X_1$ with $X_i\in\kA_i$ and by Lemma \ref{lemma:DrinfeldRelationMultiplicative} it suffices to show the relations
\begin{align}\label{eq:claim1}
\mathrm{R}(X_0,Y^{(0)})\qquad\mathrm{ and }\qquad \mathrm{R}(X_1,Y^{(1)})
\end{align}
for every short exact sequence
\begin{align*}
0\longrightarrow Y^{(1)} \longrightarrow Y \longrightarrow Y^{(0)} \longrightarrow 0
\end{align*}
in $\kA$. We automatically have $\rk(Y^{(0)})=0=\rk(Y^{(1)})$ and the relations \eqref{eq:claim1} both hold by assumption since $\mathit{Tor}(\XX)\subseteq \kA_1$. The same argument works in the case $\rk(X)=0$. This proves Claim 1.
\\ \underline{Claim 2:} $\mathrm{R}([X],[Y])$ holds for all $X,Y\in\kA$. \\
We use induction on $\rk(X)+\rk(Y)$ and assume that the statement holds if $\rk(X)+\rk(Y)\leq r-1$. Let $X,Y\in\kA$ with $\rk(X)+\rk(Y)=r$. First, we consider the special case $X\in\kA_0$.
We decompose $Y=Y_0\oplus Y_1$ with $Y_i\in\kA_i$ such that by Lemma \ref{lemma:DrinfeldRelationMultiplicative} it suffices to prove the relations
\begin{align}\label{eq:claim2}
\mathrm{R}(X^{(0)},Y_0)\qquad\mathrm{ and }\qquad \mathrm{R}(X^{(1)},Y_1)
\end{align}
for every short exact sequence
\begin{align}\label{eq:ses}
0\longrightarrow X^{(1)} \longrightarrow X \longrightarrow X^{(0)} \longrightarrow 0
\end{align}
in $\kA$. Since $\Hom_\kA(\kA_1,\kA_0)=0$ we have $X^{(1)}\in\kA_0$ which implies the second relation of \eqref{eq:claim2}. If $X^{(0)}=X$, then the first relation of \eqref{eq:claim2} holds by assumption. If $X^{(0)}\neq X$ then $X^{(1)}\neq 0$, hence $\rk(X^{(1)})>0$ because $\mathit{Tor}(\XX)\subseteq\kA_1$. It follows that $\rk(X^{(0)})<\rk(X)$ and $\mathrm{R}(X^{(0)},Y_0)$ holds by the induction hypothesis. This proves the inductive step for the special case $X\in\kA_0$.
Second, we consider the special case $X\in\kA_1$. We decompose $Y=Y_0\oplus Y_1$ with $Y_i\in\kA_i$. By Lemma \ref{lemma:DrinfeldRelationMultiplicative} it suffices to prove the relations \eqref{eq:claim2} for every short exact sequence \eqref{eq:ses} in $\kA$. This time, the vanishing of $\Hom_\kA(\kA_1,\kA_0)$ yields $X^{(0)}\in\kA_1$ such that the first relation of \eqref{eq:claim2} is satisfied by assumption. If $\rk(Y_1)<\rk(Y)$, then the second relation of \eqref{eq:claim2} follows from the induction hypothesis. If $\rk(Y_1)=\rk(Y)$, then the assumption $\mathit{Tor}(\XX)\subseteq\kA_1$ implies $Y_0=0$ and $Y^{(1)}=Y$. Thus, the second relation of \eqref{eq:claim2} is also clear. This proves the inductive step in the special case $X\in\kA_1$. Finally, we consider the general case. We decompose $X=X_0\oplus X_1$ with $X_i\in\kA_i$ and by Lemma \ref{lemma:DrinfeldRelationMultiplicative} it suffices to verify the relations
\begin{align}\label{eq:claim2general}
\mathrm{R}(X_0,Y^{(0)})\qquad\mathrm{ and }\qquad \mathrm{R}(X_1,Y^{(1)})
\end{align}
for every short exact sequence
\begin{align*}
0\longrightarrow Y^{(1)} \longrightarrow Y \longrightarrow Y^{(0)} \longrightarrow 0
\end{align*}
in $\kA$. But now the two relations in \eqref{eq:claim2general} follow from the two special cases that we have checked before. Hence $\mathrm{R}([X],[Y])$ is also satisfied in this case which completes the inductive step and the proof of the assertion of the proposition.
\end{proof}

We proceed with a few other consequences of \eqref{eq:orthogonal}, that will be useful later on.
\begin{lemma}\label{lemma:locatingExtensionsKernelsCokernels}
\begin{enumerate}
\item[(1)]
Let
\begin{align*}
0\longrightarrow Y \longrightarrow Z \longrightarrow X \longrightarrow 0
\end{align*}
be a short exact sequence in $\kA$ with $Y\in\kA_0$ and $X\in \kA_1$. Then $Z=Z_0\oplus Z_1$ for some $Z_0\in\kA_0$ and $Z_1\in\kA_1$.
\item[(2)]
Let $\psi\in\Hom_\kA(Y, X)$ be a morphism with $X,Y\in\kA_0$. Then this gives rise to an exact sequence
\begin{align*}
0\longrightarrow N_{-1}\oplus N_0 \longrightarrow Y \stackrel{\psi}\longrightarrow X \longrightarrow L_0\oplus L_1\longrightarrow 0
\end{align*}
with $N_{-1}\in\kA_{-1}$, $N_0,L_0\in\kA_0$ and $L_1\in\kA_1$.
\item[(3)]
Let $\psi\in\Hom_\kA(Y, X)$ be a morphism with $X\in\kA_0$ and $Y\in\kA_1$. Then this gives rise to an exact sequence
\begin{align*}
0\longrightarrow N \longrightarrow Y \stackrel{\psi}\longrightarrow X \longrightarrow L \longrightarrow 0
\end{align*}
with $N\in\kA_0$ and $L\in\kA_1$.
\end{enumerate}
\end{lemma}
\begin{proof}
Parts (1) and (2) directly follow from \eqref{eq:orthogonal}. In (3), the same argument yields $N=N_{-1}\oplus N_0$ for some $N_{-1}\in\kA_{-1}$ and $N_0\in\kA_0$. Since $\kA$ is hereditary, the cone of the map $Y\longrightarrow X$ is isomorphic to $N[1]\oplus L$. Hence we obtain a distinguished triangle
\begin{align*}
X \longrightarrow N_{-1}[1]\oplus N_0[1]\oplus L \longrightarrow Y[1] \longrightarrow X[1]
\end{align*}
in $\kD^b(\kA)$. From
\begin{align*}
\Hom_{\kD^b(\kA)}(X, N_{-1}[1])
\cong  
\Hom_{\kD^b(\kB)}(\underbrace{F(X)}_{\in\kB[1]}, \underbrace{F(N_{-1})[1]}_{\in\kB}) = 0
\end{align*}
one can deduce that $N_{-1}$ is a direct summand of $Y$ (see \cite[Lemma 2.5]{PengXiao_RootCategoriesAndSimpleLieAlgebras}). This shows that $N_{-1}=0$ and hence $N=N_0\in\kA_0$. A similar argument shows that $L\in\kA_1$.
\end{proof}

\section{Cramer's isomorphism}

In the following, we give a proof of Cramer's result about the homomorphism between Drinfeld doubles induced by a derived equivalence \cite[Proposition 5]{Cramer_DoubleHallAlgebras}. We will focus on those parts of the proof that need to be revised in light of the corrected relation \eqref{eq:relationDrinfeld} (replacing \cite[(5.10/5.11)]{Cramer_DoubleHallAlgebras}). We point out that the overall strategy of Cramer's proof is preserved.  

\begin{proposition}\label{prop:algebraHomomorphism}
Let $\kA,\kB$ be hereditary categories as before. Suppose we have an equivalence of triangulated categories $F:\kD^b(\kA)\longrightarrow \kD^b(\kB)$ such that the conditions from Proposition \ref{prop:reductionOfRelations} are satisfied. We define
\begin{align*}
F_*(1 \otimes [X]) &= 
\begin{cases}
\langle B_X, B_X \rangle^i (1 \otimes [B_X] k_{B_X}^i), & i \text{ even}, \\
\langle B_X, B_X \rangle^i ([B_X] k_{B_X}^i \otimes 1), & i \text{ odd},
\end{cases}\\
F_*([X] \otimes 1) &=
\begin{cases}
\langle B_X, B_X \rangle^i ([B_X] k_{B_X}^i \otimes 1), & i \text{ even}, \\
\langle B_X, B_X \rangle^i (1 \otimes [B_X] k_{B_X}^i), & i \text{ odd},
\end{cases}\\
F_*(k_X \otimes 1) &=
\begin{cases}
k_{B_X} \otimes 1, & i \text{ even}, \\
1 \otimes k_{B_X}, & i \text{ odd},
\end{cases}
\end{align*}
where $X \in \kA_i$ and $B_X = F(X)[-i]\in\kB_i\subseteq\kB$. Then this extends to an algebra homomorphism $F_*:\sD\sH_\kA\longrightarrow \sD\sH_\kB$.
\end{proposition}

The proof of Proposition \ref{prop:algebraHomomorphism} is the subject of the remainder of this article. In order to simplify the subsequent computations, we start with the following auxiliary lemma.

\begin{lemma}\label{lemma:shiftFunctor}
Proposition \ref{prop:algebraHomomorphism} holds true for the shift functor
\begin{align*}
F=[\,i\,]:~\kD^b(\kA)\longrightarrow\kD^b(\kA)
\end{align*}
for all $i\in\ZZ$.
\end{lemma}
\begin{proof}
We observe that $\kA_i=\kA$ and $B_X=X$ for all objects $X\in\kA$ in the notation of Proposition \ref{prop:algebraHomomorphism}. The relations \eqref{eq:relationLeftRight}--\eqref{eq:relationHallProductLeft} are easy to verify. We therefore focus on relation \eqref{eq:relationDrinfeld} and distinguish two cases:
\\ \underline{Case 1 ($i$ even):}
If we apply $F_*$ to the left-hand side of \eqref{eq:relationDrinfeld} we get
\begin{align*}
&F_*(\mathsf{LHS}_{(X,Y)})
=\frac{\langle X, X \rangle}{\langle Y, X \rangle}\sum_{[L], [N]} \frac{\langle N, L \rangle}{\langle L, L \rangle} \frac{a_L a_N}{a_Xa_Y} c_{Y, X}^{N[1] \oplus L}  
\langle L,L \rangle^i \langle N,N \rangle^i ([L]k_L^i \otimes [N]k_N^ik_{L-X}) \\
&\quad=\frac{\langle X, X \rangle}{\langle Y, X \rangle}\sum_{[L], [N]} \frac{\langle N, L \rangle}{\langle L, L \rangle} \frac{a_L a_N}{a_Xa_Y} c_{Y, X}^{N[1] \oplus L}  
\underbrace{\frac{\langle L,L \rangle^i \langle N,N \rangle^i}{(L,N)^i}}_{=\langle N-L,N-L\rangle^i} ([L] \otimes [N]k_{N-L}^ik_{L-X}) \\
&\quad=\langle Y-X,Y-X\rangle^i \cdot\frac{\langle X, X \rangle}{\langle Y, X \rangle}\sum_{[L], [N]} \frac{\langle N, L \rangle}{\langle L, L \rangle} \frac{a_L a_N}{a_Xa_Y} c_{Y, X}^{N[1] \oplus L}  
([L] \otimes [N]k_{L-X})(1\otimes k_{Y-X}^i) \\
&\quad=\langle Y-X,Y-X\rangle^i \cdot \mathsf{LHS}_{(X,Y)} \cdot (1\otimes k_{Y-X}^i)
\end{align*}
where we used the identity $N-L=Y-X$ in $\sK_0(\kA)$.
A similar computation shows that $F_*(\mathsf{RHS}_{(X,Y)}) = \langle Y-X,Y-X\rangle^i \cdot \mathsf{RHS}_{(X,Y)} \cdot (1\otimes k_{Y-X}^i)$ which proves that $F_*$ preserves the relation \eqref{eq:relationDrinfeld}.
\\ \underline{Case 2 ($i$ odd):}
Applying $F_*$ to the left-hand side of \eqref{eq:relationDrinfeld} yields
\begin{align*}
&F_*(\mathsf{LHS}_{(X,Y)})
=\frac{\langle X,X \rangle}{\langle Y, X \rangle}\sum_{[L], [N]} \frac{\langle N, L \rangle}{\langle L, L \rangle} \frac{a_L a_N}{a_Xa_Y} c_{Y, X}^{N[1] \oplus L} 
\langle L,L\rangle^i \langle N,N\rangle^i (1\otimes [L]k_L^i)([N]k_N^i\otimes k_{X-L})\\[+5pt]
&\quad=\frac{\langle X,X \rangle}{\langle Y, X \rangle}\sum_{[L], [N]} \frac{\langle N, L \rangle}{\langle L, L \rangle} \frac{a_L a_N}{a_Xa_Y} c_{Y, X}^{N[1] \oplus L} 
\underbrace{\frac{\langle L,L\rangle^i \langle N,N\rangle^i}{(L,N)^i}}_{=\langle N-L,N-L\rangle^i}
(1\otimes [L])([N]k_{N-L}^i\otimes \underbrace{k_{X-L}}_{=k_{Y-N}})\\[+5pt]
&\quad=\langle Y-X,Y-X\rangle^i \cdot\frac{\langle X, X \rangle}{\langle Y, X \rangle}\sum_{[L], [N]} \frac{\langle N, L \rangle}{\langle L, L \rangle} \frac{a_L a_N}{a_Xa_Y} c_{Y, X}^{N[1] \oplus L}
(1\otimes [L])([N]\otimes k_{Y-N})(k_{Y-X}^i\otimes 1)\\[+5pt]
&\quad=\langle Y-X,Y-X\rangle^i \cdot \mathsf{RHS}_{(Y,X)} \cdot(k_{Y-X}^i \otimes 1).
\end{align*}
Here we identified the middle part of the expression with the right-hand side of relation \eqref{eq:relationDrinfeld}, but with $X$ and $Y$ interchanged. Analogously, one can compute that
\begin{align*}
F_*(\mathsf{RHS}_{(X,Y)}) = \langle Y-X,Y-X\rangle^i \cdot \mathsf{LHS}_{(Y,X)} \cdot(k_{Y-X}^i \otimes 1)
\end{align*}
which shows $F_*(\mathsf{LHS}_{(X,Y)})=F_*(\mathsf{RHS}_{(X,Y)})$ and finishes the proof.
\end{proof}

\noindent
In order to prove Proposition \ref{prop:algebraHomomorphism} we need to show that the function $F_*$ respects the relations \eqref{eq:relationLeftRight}--\eqref{eq:relationDrinfeld} for all $X\in\kA_i$, $Y\in\kA_j$. We observe:

\begin{itemize}
\item The relations \eqref{eq:relationLeftRight}--\eqref{eq:relationK0reduced} are clearly preserved by $F_*$.
\item The relations \eqref{eq:relationK0swapRight} and \eqref{eq:relationK0swapLeft} were established by Cramer \cite{Cramer_DoubleHallAlgebras}.
\item The relation \eqref{eq:relationHallProductRight} is covered by \cite{Cramer_DoubleHallAlgebras} in the cases $|i-j|\geq 2$, $j=i+1$ and $i=j$.
\item The relation \eqref{eq:relationDrinfeld} is clearly preserved for $|i-j|\geq 2$.
\end{itemize}

\noindent
Therefore it remains to show:
\begin{itemize}
\item The relation \eqref{eq:relationHallProductRight} is preserved in the case $j=i-1$ (note that \eqref{eq:relationHallProductLeft} follows from an analogous computation).
\item The relation \eqref{eq:relationDrinfeld} is preserved for $i=j$.
\item The relation \eqref{eq:relationDrinfeld} is preserved for $|i-j|= 1$.
\end{itemize}

\subsection*{Relation \eqref{eq:relationHallProductRight} for $j=i-1$}
\phantom{.}\\
By Lemma \ref{lemma:shiftFunctor} we can assume without loss of generality that $Y\in\kA_0$ and $X\in\kA_1$. In Lemma \ref{lemma:locatingExtensionsKernelsCokernels} (1) we have already seen that the middle term of a short exact sequence $0\rightarrow Y\rightarrow Z\rightarrow X\rightarrow 0$ is contained in $\kA_0\vee\kA_1$, hence the right-hand side of \eqref{eq:relationHallProductRight} can be rewritten as
\begin{align*}
\mathsf{RHS}=\langle X,Y \rangle \sum_{[Z]} \frac{P_{X,Y}^Z}{a_Xa_Y} (1 \otimes [Z])
&= \langle X,Y \rangle \sum_{[Z_0],[Z_1]} \frac{P_{X,Y}^{Z_0\oplus Z_1}}{a_Xa_Y} (1 \otimes [Z_0\oplus Z_1])\\
&= \langle X,Y \rangle \sum_{[Z_0],[Z_1]} \frac{P_{X,Y}^{Z_0\oplus Z_1}}{a_Xa_Y}\frac{1}{\langle Z_0,Z_1\rangle} (1 \otimes [Z_0][Z_1])
\end{align*}
where the sum runs over all isomorphism classes of objects $Z_0\in\kA_0$ and $Z_1\in\kA_1$.
Let us write $B_Y=F(Y)$, $N=F(Z_0)$, $L=F(Z_1)[-1]$ and $B_X=F(X)[-1]\in\kB$.
If we apply $F_*$ to the expression above, we obtain
\begin{align*}
F_*(\mathsf{RHS})
&=\langle X,Y \rangle \sum_{[Z_0],[Z_1]} \frac{P_{X,Y}^{Z_0\oplus Z_1}}{a_Xa_Y}\underbrace{\frac{1}{\langle Z_0,Z_1\rangle}}_{=\langle N,L\rangle} \langle L,L\rangle (1 \otimes [N])([L]k_L\otimes 1)\\
&=\langle X,Y \rangle \sum_{[Z_0],[Z_1]} \frac{P_{X,Y}^{Z_0\oplus Z_1}}{a_Xa_Y}\langle N,L\rangle\langle L,L\rangle (1 \otimes [N])([L]\otimes k_L^{-1}).
\end{align*} 
Since $\Hom_\kA(X,Y)=0$, Lemma \ref{lemma:cEqualP} yields
\begin{align*}
c_{B_X,B_Y}^{L[1]\oplus N}
&=\left|\left\{B_X\stackrel{\psi}{\longrightarrow}B_Y~\bigg\vert~\cone(\psi)\cong L[1]\oplus N\right\}\right|
=\left|\left\{X[-1]\stackrel{\overline{\psi}}{\longrightarrow}Y~\bigg\vert~\cone(\overline{\psi})\cong Z_0\oplus Z_1\right\}\right|\\[+5pt]
&= \frac{P_{X,Y}^{Z_0\oplus Z_1}}{\big|\Aut_\kA(Z_0\oplus Z_1)\big|}
= \frac{P_{X,Y}^{Z_0\oplus Z_1}}{a_La_N \big|\Hom_\kA(Z_0, Z_1)\big|}
= \frac{P_{X,Y}^{Z_0\oplus Z_1}}{a_La_N\langle Z_0,Z_1\rangle^2}
= \frac{P_{X,Y}^{Z_0\oplus Z_1}\langle N,L\rangle^2}{a_La_N}.
\end{align*}
Further, note that the summation over all isomorphism classes of objects $Z_0\in\kA_0$ and $Z_1\in \kA_1$ is equivalent to the summation over all $N\in\kB_0$ and $L\in\kB_1$. By Lemma \ref{lemma:locatingExtensionsKernelsCokernels} (3), this is equivalent to the summation over all $N,L\in\kB$ since the number $c_{B_X,B_Y}^{L[1]\oplus N}$ vanishes in all the other cases. Using the identity ${N}-{L} = {B_Y}-{B_X}$ in $\sK_0(\kB)$, we therefore obtain
\begin{align*}
F_*(\mathsf{RHS})
&=\langle X,Y \rangle \sum_{[Z_0],[Z_1]} \frac{a_La_N}{a_Xa_Y}c_{B_X,B_Y}^{L[1]\oplus N}\frac{\langle N,L\rangle\langle L,L\rangle}{\langle N,L\rangle^2} (1 \otimes [N])([L]\otimes k_L^{-1})\\[+5pt]
&=\langle X,Y \rangle \sum_{[Z_0],[Z_1]} \frac{a_La_N}{a_Xa_Y}c_{B_X,B_Y}^{L[1]\oplus N}\frac{\langle L,L\rangle}{\langle N,L\rangle} (1 \otimes [N])([L]\otimes k_L^{-1})\\[+5pt]
&=\langle X,Y \rangle \sum_{[Z_0],[Z_1]} \frac{a_La_N}{a_Xa_Y}c_{B_X,B_Y}^{L[1]\oplus N}
\frac{\langle L,N\rangle}{\langle N,N\rangle} \underbrace{\frac{\langle N,N\rangle}{\langle L,N\rangle} \frac{\langle L,L\rangle}{\langle N,L\rangle}}_{=\langle N-L,N-L\rangle}
(1 \otimes [N])([L]\otimes k_L^{-1})\\[+5pt]
&=\frac{\langle B_Y-B_X, B_Y-B_X \rangle}{\langle B_X,B_Y \rangle} \sum_{[Z_0],[Z_1]} 
\frac{\langle L,N\rangle}{\langle N,N\rangle}\frac{a_La_N}{a_Xa_Y}c_{B_X,B_Y}^{L[1]\oplus N}
(1 \otimes [N])([L]\otimes k_L^{-1}) \\[+5pt]
&=\frac{\langle B_X,B_X\rangle}{(B_X,B_Y)}\underbrace{\frac{\langle B_Y, B_Y \rangle}{\langle B_X,B_Y \rangle} 
\sum_{[Z_0],[Z_1]} 
\frac{\langle L,N\rangle}{\langle N,N\rangle}\frac{a_La_N}{a_Xa_Y}c_{B_X,B_Y}^{L[1]\oplus N}
(1 \otimes [N])([L]\otimes k_{B_X-L})}_{\hspace{4.1cm}=[B_X]\otimes [B_Y] \text{ by (\ref{eq:relationDrinfeld}) using} \Hom_{\kB}(B_Y, B_X) = 0}(1\otimes k_{B_X}^{-1})\\[+5pt]
&=\frac{\langle B_X,B_X\rangle}{(B_X,B_Y)}([B_X]\otimes [B_Y] k_{B_X}^{-1})
=\langle B_X,B_X\rangle([B_X]k_{B_X}\otimes [B_Y])\\[+5pt]
&=\langle B_X,B_X\rangle([B_X]k_{B_X}\otimes 1)(1\otimes [B_Y])
=F_*([B_X]\otimes 1)\cdot F_*(1\otimes [B_Y])
\end{align*}
This finishes the proof that $F_*$ preserves the relation \eqref{eq:relationHallProductRight}.

\subsection*{Relation \eqref{eq:relationDrinfeld} for $i=j$}
\phantom{.}\\
By Lemma \ref{lemma:shiftFunctor} we can assume without loss of generality that $X,Y\in\kA_0$.
The left-hand side of \eqref{eq:relationDrinfeld} reads 
\begin{align*}
\mathsf{LHS}_{(X,Y)}=\frac{\langle X, X \rangle}{\langle Y, X \rangle}\sum_{[L], [N]} \frac{\langle N, L \rangle}{\langle L, L \rangle} \frac{a_La_N}{a_Xa_Y} c_{Y, X}^{N[1] \oplus L}  ([L] \otimes [N]k_{L-X}).
\end{align*}
By Lemma \ref{lemma:locatingExtensionsKernelsCokernels} (2) any map $Y\longrightarrow X$ gives rise to an exact sequence
\begin{align*}
0\longrightarrow \underbrace{N_{-1}\oplus N_0}_{=N}\longrightarrow Y \longrightarrow X \longrightarrow \underbrace{L_0\oplus L_1}_{=L} \longrightarrow 0
\end{align*}
with $N_{-1}\in\kA_{-1}$, $N_0,L_0\in\kA_0$ and $L_1\in\kA_1$.
We write $B_{N_{-1}}=F(N_{-1})[1]$, $B_{N_0}=F(N_0)$, $B_Y=F(Y)$, $B_X=F(X)$, $B_{L_0}=F(L_0)$ and $B_{L_1}=F(L_1)[-1]\in\kB$ and observe that
\begin{align}
\nonumber
c_{Y,X}^{N[1]\oplus L} 
&= \left|\left\{Y\stackrel{\psi}{\longrightarrow}X~\bigg\vert~\cone(\psi)\cong (N_{-1}\oplus N_0)[1]\oplus (L_0\oplus L_1)\right\}\right| \\ \nonumber
&= \left|\left\{B_Y\stackrel{\overline{\psi}}{\longrightarrow}B_X~\bigg\vert~\cone(\overline{\psi})\cong B_{N_{-1}}\oplus B_{N_0}[1]\oplus B_{L_0}\oplus B_{L_1}[1]\right\}\right| \\ 
&= \left|\left\{B_Y\stackrel{\overline{\psi}}{\longrightarrow}B_X~\bigg\vert~\cone(\overline{\psi})\cong N'[1]\oplus L'\right\}\right| 
= c_{B_Y,B_X}^{N'[1]\oplus L'}
\label{eq:cEqualc}
\end{align}
where we have set
\begin{align*}
N'=B_{N_0} \oplus B_{L_1} \qquad\text{ and }\qquad L'=B_{L_0} \oplus B_{N_{-1}}.
\end{align*}
The next step is to compute the automorphism groups of the objects $N',L'\in\kB$ and compare them with those of $N,L\in\kA$. We have
\begin{align*}
\Aut_\kA(N)&=\begin{pmatrix} \Aut_\kA(N_{-1}) & 0 \\ \Hom_\kA(N_{-1},N_0) & \Aut_\kA(N_0) \end{pmatrix},\quad
\Aut_\kA(L)=\begin{pmatrix} \Aut_\kA(L_0) & 0 \\ \Hom_\kA(L_0,L_1) & \Aut_\kA(L_1) \end{pmatrix},\\[+5pt]
\Aut_\kB(N')&=\begin{pmatrix} \Aut_\kA(N_0) & \Ext^1_\kA(L_1,N_0) \\ 0 & \Aut_\kA(L_1) \end{pmatrix},\qquad
\Aut_\kB(L')=\begin{pmatrix} \Aut_\kA(L_0) & 0 \\ \Ext^1_\kA(L_0,N_{-1}) & \Aut_\kA(N_{-1}) \end{pmatrix}.
\end{align*}
Using the $\Hom$- and $\Ext$-vanishing properties \eqref{eq:orthogonal}, we obtain
\begin{align*}
\frac{a_La_N}{a_{L'}a_{N'}}&=\frac{\big|\Hom_\kA(N_{-1},N_0)\big|\cdot\big|\Hom_\kA(L_0,L_1)\big|}{\big|\Ext^1_\kA(L_1,N_0)\big|\cdot\big|\Ext^1_\kA(L_0,N_{-1})\big|}
=\frac{\langle N_{-1},N_0 \rangle^2\cdot\langle L_0,L_1 \rangle^2}{\langle L_1,N_0 \rangle^{-2}\cdot \langle L_0,N_{-1} \rangle^{-2}}\\[+5pt]
&=\langle N_{-1},N_0 \rangle^2\langle L_0,L_1 \rangle^2 \langle L_1,N_0 \rangle^2 \langle L_0,N_{-1} \rangle^2.
\end{align*}
and hence
\begin{align}
\label{eq:autNautL}
\frac{a_La_N}{a_Xa_Y}
=\langle N_{-1},N_0 \rangle^2\langle L_0,L_1 \rangle^2 \langle L_1,N_0 \rangle^2 \langle L_0,N_{-1} \rangle^2 \cdot \frac{a_{L'}a_{N'}}{a_{B_X}a_{B_Y}}.
\end{align}
We proceed with the next preparatory calculation:
\begin{align*}
\frac{\langle N,L\rangle}{\langle L,L\rangle}\cdot\frac{\langle L',L'\rangle}{\langle N',L'\rangle}
&=\frac{\langle N_0,L_0\rangle\langle N_0,L_1\rangle\langle N_{-1},L_0\rangle\overbrace{\langle N_{-1},L_1\rangle}^{=1}}{\langle L_0,L_0\rangle\langle L_0,L_1\rangle\langle L_1,L_0\rangle\langle L_1,L_1\rangle}
\cdot \\ &\qquad \cdot
\frac{\langle B_{L_0},B_{L_0}\rangle\langle B_{L_0},B_{N_{-1}}\rangle\langle B_{N_{-1}},B_{L_0}\rangle\langle B_{N_{-1}},B_{N_{-1}}\rangle}{\langle B_{N_0},B_{L_0}\rangle\langle B_{N_0},B_{N_{-1}}\rangle\langle B_{L_1},B_{L_0}\rangle\underbrace{\langle B_{L_1},B_{N_{-1}}\rangle}_{=1}}\\
&=\frac{\langle N_0,L_1\rangle\langle N_{-1},N_{-1}\rangle\langle N_0,N_{-1}\rangle}{\langle L_0,L_1\rangle\langle L_1,L_1\rangle\langle L_0,N_{-1}\rangle}
\end{align*}
If we combine this with the identities
\begin{align*}
[L] = [L_0\oplus L_1] =\frac{1}{\langle L_0,L_1\rangle}[L_0][L_1]
\quad\text{ and }\quad
[N] = [N_{-1}\oplus N_0] = \frac{1}{\langle N_{-1},N_0\rangle}[N_{-1}][N_0],
\end{align*}
we obtain
\begin{align*}
&\frac{\langle N,L\rangle}{\langle L,L\rangle}  F_*([L]\otimes[N]k_{L-X})
=\frac{\langle N,L\rangle}{\langle L,L\rangle}\frac{1}{\langle N_{-1},N_0\rangle\langle L_0,L_1\rangle}  F_*([L_0][L_1]\otimes[N_{-1}][N_0]k_{L_0+L_1-X})\\[+5pt]
&\quad =\frac{\langle L',L'\rangle}{\langle N',L'\rangle}
\frac{\langle N_0,L_1\rangle\langle N_{-1},N_{-1}\rangle\langle N_0,N_{-1}\rangle}{\langle L_0,L_1\rangle^2\langle L_1,L_1\rangle\langle L_0,N_{-1}\rangle\langle N_{-1},N_0\rangle}
  F_*([L_0][L_1]\otimes[N_{-1}][N_0]k_{L_0+L_1-X})\\[+5pt]
&\quad =\frac{\langle L',L'\rangle}{\langle N',L'\rangle}
\frac{\langle N_0,L_1\rangle\langle N_0,N_{-1}\rangle}{\langle L_0,L_1\rangle^2\langle L_0,N_{-1}\rangle\langle N_{-1},N_0\rangle}
  ([B_{L_0}]\otimes[B_{L_1}]k_{B_{L_1}})([B_{N_{-1}}]k_{B_{N_{-1}}}^{-1}\otimes[B_{N_0}]k_{B_{L_0}-B_{L_1}-B_X})\\[+5pt]
&\quad =\frac{\langle L',L'\rangle}{\langle N',L'\rangle}
\frac{\langle N_0,L_1\rangle\langle N_0,N_{-1}\rangle}{\langle L_0,L_1\rangle^2\langle L_0,N_{-1}\rangle\langle N_{-1},N_0\rangle}
  ([B_{L_0}][B_{N_{-1}}]\otimes[B_{L_1}]k_{B_{L_1}}k_{B_{N_{-1}}}[B_{N_0}]k_{B_{L_0}-B_{L_1}-B_X})
\end{align*}
\begin{align*}
&\quad =\frac{\langle L',L'\rangle}{\langle N',L'\rangle}
\frac{\langle N_0,L_1\rangle}{\langle L_0,L_1\rangle^2\langle L_0,N_{-1}\rangle\langle N_{-1},N_0\rangle^2}
  ([B_{L_0}][B_{N_{-1}}]\otimes[B_{L_1}]k_{B_{L_1}}[B_{N_0}]k_{B_{N_{-1}}}k_{B_{L_0}-B_{L_1}-B_X})\\[+5pt]
&\quad =\frac{\langle L',L'\rangle}{\langle N',L'\rangle}
\frac{1}{\langle L_0,L_1\rangle^2\langle L_0,N_{-1}\rangle\langle N_{-1},N_0\rangle^2\langle L_1,N_0\rangle}
  ([B_{L_0}][B_{N_{-1}}]\otimes[B_{L_1}][B_{N_0}]k_{B_{N_{-1}}+B_{L_0}-B_X})\\[+5pt]
&\quad =\frac{\langle L',L'\rangle}{\langle N',L'\rangle}
\frac{1}{\langle L_0,L_1\rangle^2\langle L_0,N_{-1}\rangle^2\langle N_{-1},N_0\rangle^2\langle L_1,N_0\rangle^2}
  ([B_{L_0}\oplus B_{N_{-1}}]\otimes[B_{L_1}\oplus B_{N_0}]k_{L'-B_X})\\[+5pt]
&\quad =\frac{\langle L',L'\rangle}{\langle N',L'\rangle}
\frac{1}{\langle L_0,L_1\rangle^2\langle L_0,N_{-1}\rangle^2\langle N_{-1},N_0\rangle^2\langle L_1,N_0\rangle^2}
  ([L']\otimes[N']k_{L'-B_X}).
\end{align*}
Using \eqref{eq:cEqualc} and \eqref{eq:autNautL} we arrive at
\begin{align*}
\frac{\langle N,L\rangle}{\langle L,L\rangle}\frac{a_La_N}{a_Xa_Y}c_{Y,X}^{N[1]\oplus L}
 F_*([L]\otimes[N]k_{L-X})
=\frac{\langle L',L'\rangle}{\langle N',L'\rangle}\frac{a_{L'}a_{N'}}{a_{B_X}a_{B_Y}}c_{B_Y,B_X}^{N'[1]\oplus L'}
 ([L']\otimes[N']k_{L'-B_X}).
\end{align*}
In view of Lemma \ref{lemma:locatingExtensionsKernelsCokernels} (2), taking the sum over all (isomorphism classes of) objects $N,L\in\kA$ is equivalent to taking the sum over all $N_{-1}\in\kA_{-1}$, $N_0,L_0\in\kA_0$ and $L_1\in\kA_1$. Applying $F$, this is equivalent to taking the sum over all $B_{N_{-1}}\in\kB_{-1}$, $B_{N_0},B_{L_0}\in\kB_0$ and $B_{L_1}\in\kB_1$. Using Lemma \ref{lemma:locatingExtensionsKernelsCokernels} (2) in the category $\kB$, we can deduce that every map $B_Y\longrightarrow B_X$ has kernel in $\kB_1\vee\kB_0$ and cokernel in $\kB_0\vee\kB_{-1}$. Hence this is also equivalent to summation over all (isomorphism classes of) objects $N',L'\in\kB$. We therefore obtain
\begin{align*}
&F_*\left(\frac{\langle X,X\rangle}{\langle Y,X\rangle}\sum_{[L][N]}\frac{\langle N,L\rangle}{\langle L,L\rangle}\frac{a_La_N}{a_Xa_Y}c_{Y,X}^{N[1]\oplus L}
([L]\otimes[N]k_{L-X})\right)
\\ &\hspace{3cm}
=\frac{\langle B_X,B_X\rangle}{\langle B_Y,B_X\rangle}\sum_{[L'][N']}\frac{\langle L',L'\rangle}{\langle N',L'\rangle}\frac{a_{L'}a_{N'}}{a_{B_X}a_{B_Y}}c_{B_Y,B_X}^{N'[1]\oplus L'}
  ([L']\otimes[N']k_{L'-B_X})
\end{align*}
and have proved the first equality in the following chain:
\begin{align*}
F_*(\mathsf{LHS}_{(X,Y)})
=\mathsf{LHS}_{(B_X,B_Y)}
=\mathsf{RHS}_{(B_X,B_Y)}
=F_*(\mathsf{RHS}_{(X,Y)})
\end{align*}
The second equality is the relation \eqref{eq:relationDrinfeld} in $\sD\sH_\kB$ while the third equality can be checked via a similar computation as the first one. This proves that $F_*$ preserves \eqref{eq:relationDrinfeld} in the case $X,Y\in\kA_i$.

\subsection*{Relation \eqref{eq:relationDrinfeld} for $|i-j|=1$}
\phantom{.}\\ 
We will assume that $Y\in\kA_0, X\in\kA_1$. For $X\in\kA_0, Y\in\kA_1$ the computations are similar and by Lemma \ref{lemma:shiftFunctor}, it suffices to consider these two cases.
As already pointed out in Remark \ref{remark:SimplifiedRelationHomVanishing}, the relation \eqref{eq:relationDrinfeld} takes the simpler form
\begin{align}
\label{eq:DDrelationCase2}
\frac{\langle X, X \rangle}{\langle Y, X \rangle}\sum_{[L], [N]} \frac{\langle N, L \rangle}{\langle L,L \rangle} \frac{a_L a_N}{a_Xa_Y} c_{Y, X}^{N[1] \oplus L}  ([L]\otimes [N]k_{L-X})
=(1\otimes [Y])([X] \otimes 1)
\end{align}
because $\Hom_\kA(X,Y)=0$.
By Lemma \ref{lemma:locatingExtensionsKernelsCokernels} (3), every morphism $Y\longrightarrow X$ gives rise to an exact sequence
\begin{align*}
0\longrightarrow N \longrightarrow Y \longrightarrow X \longrightarrow L\longrightarrow 0
\end{align*}
in $\kA$ with $N\in\kA_0$ and $L\in\kA_1$.
We therefore write $B_N=F(N)$, $B_Y=F(Y)$, $B_X=F(X)[-1]$ and $B_L=F(L)[-1]\in\kB$. In particular, we get $B_L+B_N=B_X+B_Y$ in $\sK_0(\kB)$. We start by computing
\begin{align*}
F_*([L]\otimes [N]k_{L-X}) &= \langle B_L,B_L\rangle(1\otimes [B_L]k_{B_L}[B_N]k_{B_X-B_L}) \\[+5pt]
&= \langle B_L,B_L\rangle (B_L,B_N) (1\otimes [B_L][B_N]k_{B_L}k_{B_X-B_L}) \\[+5pt]
&= \langle B_L,B_L\rangle \langle B_N,B_L\rangle \langle B_L,B_N\rangle^2(1\otimes [B_L\oplus B_N]k_{B_X}) \\[+5pt]
&= \frac{\langle L,L\rangle}{\langle N,L\rangle \langle L,N\rangle^2}(1\otimes [B_L\oplus B_N]k_{B_X}).\\
\end{align*}
Since $\Hom_\kB(B_Y,B_X)=0$, Lemma \ref{lemma:cEqualP} yields
\begin{align*}
c_{Y,X}^{N[1]\oplus L}
&=c_{F(Y),F(X)}^{F(N[1]\oplus L)}
=\left|\left\{B_Y\stackrel{\psi}{\longrightarrow}B_X[1]~\bigg\vert~\cone(\psi)\cong B_N[1]\oplus B_L[1]\right\}\right|\\[+5pt]
&= \dfrac{P_{B_Y,B_X}^{B_N\oplus B_L}}{\big|\Aut_{\kB}(B_N\oplus B_L)\big|}
= \dfrac{P_{B_Y,B_X}^{B_N\oplus B_L}}{\big|\Aut_{\kD^b(\kA)}(N[1]\oplus L)\big|}
= \dfrac{P_{B_Y,B_X}^{B_N\oplus B_L}}{a_La_N\big|\Ext_\kA^1(L,N)\big|}
= \dfrac{P_{B_Y,B_X}^{B_N\oplus B_L}\langle L,N\rangle^2}{a_La_N}.\\
\end{align*}
When we apply $F_*$ to the left-hand side of \eqref{eq:DDrelationCase2}, we get
\begin{align*}
&\frac{\langle X, X \rangle}{\langle Y, X \rangle}\sum_{[L], [N]} \frac{\langle N, L \rangle}{\langle L,L \rangle} \frac{a_L a_N}{a_Xa_Y} c_{Y, X}^{N[1] \oplus L}  F_*([L]\otimes [N]k_{L-X}) \\[+5pt]
&\hspace{2cm}= \frac{\langle X, X \rangle}{\langle Y, X \rangle}\sum_{[B_L], [B_N]} \frac{P_{B_Y,B_X}^{B_N\oplus B_L}}{a_Xa_Y} (1\otimes [B_L\oplus B_N]k_{B_X}) \\[+5pt]
&\hspace{2cm}= \langle B_X, B_X \rangle \underbrace{\langle B_Y, B_X \rangle \sum_{[B]} \frac{P_{B_Y,B_X}^{B}}{a_{B_X}a_{B_Y}} (1\otimes [B])}_{= 1\otimes [B_Y][B_X]}(1\otimes k_{B_X}) \\[+5pt]
&\hspace{2cm}= (1\otimes[B_Y]) \cdot \langle B_X,B_X\rangle(1\otimes [B_X] k_{B_X}) \\[+5pt] 
&\hspace{2cm}
=F_*(1\otimes [Y])\cdot F_*([X]\otimes 1),
\end{align*}
which coincides with the image of the right-hand side of \eqref{eq:DDrelationCase2} under $F_*$. We have used that -- by Lemma \ref{lemma:locatingExtensionsKernelsCokernels} (3) -- the summation over all (isomorphism classes of) objects $L,N\in\kA$ is equivalent to the summation over all $N\in\kA_0$ and $L\in\kA_1$ which is equivalent to the summation over all $B_N\in\kB_0$ and $B_L\in\kB_1$. By Lemma \ref{lemma:locatingExtensionsKernelsCokernels} (1), this in turn is equivalent to the summation over all (isomorphism classes of) objects $B\in \kB$. This finishes the proof that $F_*$ preserves the relation \eqref{eq:relationDrinfeld} in this case.\\
Moreover, this completes the proof of Proposition \ref{prop:algebraHomomorphism}.\\

\noindent
As an immediate consequence we obtain the main result of Cramer's work \cite{Cramer_DoubleHallAlgebras}.

\begin{theorem}\label{thm:CramerIsomorphism}
Let $\kA$ and $\kB$ be hereditary categories and $F: \kD^b(\kA)\longrightarrow \kD^b(\kB)$ a derived equivalence such that the conditions of Proposition \ref{prop:reductionOfRelations} are satisfied for both $\kA$ and $\kB$. Then $F$ induces an isomorphism of $\widetilde\QQ$-algebras $\sD\sH_\kA\longrightarrow \sD\sH_\kB$.
\end{theorem}
\begin{remark}
This also corrects the proof of \cite[Theorem 3.9]{BurbanSchiffmann_HallAlgebraEllipticCurve}.
\end{remark}

\end{document}